\documentclass[english]{article}
\usepackage{amssymb}
\usepackage{bm}
\usepackage{bbm}
\usepackage[T1]{fontenc}
\usepackage{enumerate}
\usepackage{geometry}
\usepackage{amsmath}
\usepackage{comment}
\usepackage{amsthm}
\usepackage{amssymb}
\usepackage{mathrsfs}
\usepackage{graphicx}
\usepackage{subfigure}
\usepackage{pifont}
\usepackage{xcolor}

\usepackage[utf8]{inputenc}
\usepackage[english]{babel}

\usepackage{amssymb}
\usepackage{amsmath}
\usepackage{amsthm}
\usepackage{leftidx}
\usepackage{mathdots}
\usepackage{amssymb}
\usepackage{mathrsfs}

\usepackage{epsfig}
\usepackage{verbatim}
\allowdisplaybreaks[4]

\makeatletter
\theoremstyle{plain}
\newtheorem{thm}{\protect\theoremname}[section]
  \theoremstyle{plain}
  \newtheorem{lem}[thm]{\protect\lemmaname}
    \theoremstyle{plain}
  \newtheorem{rem}[thm]{\protect\remarkname}
      \theoremstyle{plain}
  \newtheorem{defi}[thm]{\protect\definitionname}
        \theoremstyle{plain}
  \newtheorem{prop}[thm]{\protect\propositionname}
          \theoremstyle{plain}
  \newtheorem{coro}[thm]{\protect\corollaryname}

\makeatother

\usepackage{babel}
  \providecommand{\lemmaname}{Lemma}
\providecommand{\theoremname}{Theorem}
\providecommand{\remarkname}{Remark}
\providecommand{\definitionname}{Definition}
\providecommand{\propositionname}{Proposition}
\providecommand{\corollaryname}{Corollary}

\date{}

\begin{document}

\title{A New Linear Programming Method  in Sphere Packing}
\author{Qun Mo
\thanks{School of Mathematics,
Zhejiang University, Hangzhou, 310027, China  (E-mail: moqun@zju.edu.cn).},
Jinming Wen
\thanks{Colleges of Information Science, Jinan University, Guangzhou, 510632, China (E-mail: jinming.wen@mail.mcgill.ca).},
and Yu Xia
\thanks{School of Mathematics, Hangzhou Normal University, Hangzhou, 311121, China (E-mail: yxia@hznu.edu.cn). \protect\\
This work is supported by the NSFC grant (12271215, 12326378, 11871248, 11871481, 11971427), the key project of Zhejiang Provincial Natural Science Foundation under grant number LZ23A010002 and STI2030-Major Project 2021ZD0204105. }\\}

\maketitle

\begin{abstract}
Inspired by the linear programming method developed by Cohn and Elkies (Ann. Math. 157(2): 689-714, 2003), we introduce a new linear programming method to solve the sphere packing problem.
More concretely, when dimension $n$ is fixed, the sequences of auxiliary functions $\{g_m\}_{m\in \mathbb{N}^{+}}$ is considered, where $g_m$ is a $m\Lambda$-periodic auxiliary function defined on $\mathbb{R}^n$,
with $\Lambda$ being a  given full-rank lattice in $\mathbb{R}^n$. Using this new linear programming framework, we construct several effective auxiliary functions for dimensions $n=1,2,3$. We hope this approach provides valuable insights into solving sphere packing problems for $n=2,3$ and even higher dimensions.
\end{abstract}
{\textbf{Keywords:} Sphere packing, Periodic packing, Center density, Auxiliary function, Linear programming}

\section{Introduction}
The classical  \textit{sphere packing problem} considers the densest packing of spheres in $\mathbb{R}^n$,
that is, determine the density $\Delta_n$, which is the largest proportion of $\mathbb{R}^n$ covered by a collection of identical spheres that do not intersect except along their boundaries.
It is a natural geometric problem in its own right and arises from numerous applications, such as error-correcting codes design for real world communication channels \cite{SP1,SP2,SP3}.
Determining the value of the densest packing is an old and  challenging problem and is only known for  $n\in \{1,2,3,8,24\}$.
The sphere packing problem for $n=1$ is trivial. In fact, we can take consecutive line segments that cover $\mathbb{R}$ and the packing density is $1$. This is the largest packing density for all $n\in \mathbb{N}^{+}$.
The problem for $n=2$ is already non-trivial and was proved by Thue \cite{22} in 1892.
The problem for $n=3$ is the Kepler conjecture raised in 1611 and was resolved by Hales \cite{14} with some computer-assisted proof and formally verified in 2017 \cite{15}.
On the first three dimensions, guessing the optimal packing is quite easy. Unfortunately, the low-dimensional experience is too optimistic for understanding high-dimensional sphere packing problem.

In 2017, an inspiring breakthrough is that Viazovska \cite{Viazovska} and  {Cohn  et al.} \cite{Viazovska1} solved the sphere packing problem for $n=8$ and $n=24$, respectively.
To this day, $n=8$ and $n=24$ are the only two cases beyond $n=3$ that have been solved.
The main idea to solve the sphere packing problem in $n=8$ and $n=24$ is applying the linear programming technique developed by Cohn and Elkies \cite{Cohn}.
They defined a class of function called admissible functions and proved that, if an admissible function $f:\mathbb{R}^n\rightarrow \mathbb{R}$ satisfies certain specific conditions,
then the center density $\delta_n$, which is $\frac{\Gamma(n/2+1)}{\pi^{n/2}}\Delta_n$, satisfies $\delta_n\leq \frac{f(\boldsymbol{0})}{2^n\widehat{f}(\boldsymbol{0})}$, where $\widehat{f}(\boldsymbol{0})=\int_{\mathbb{R}^n}f(\boldsymbol{x})d\boldsymbol{x}$.
On the basis of numerical evidence,  Cohn and Elkies \cite[Conjecture 7.3]{Cohn} conjectured the constructions of some optimal auxiliary functions in dimensions $2$, $8$ and $24$.
Viazovska \cite{Viazovska} and  {Cohn et al. }\cite{Viazovska1} found optimal auxiliary functions for $n=8$ and $n=24$ that match the density of the $E_8$ root lattice and Leech lattice, respectively.
However, to one's surprise, an optimal auxiliary function in $2$-dimensional case is still unknown until now.
On the other hand, the linear programming bound developed by Cohn and Elkies cannot come near the best known packing densities in $n\in \{12,16,20,28,32\}$ \cite{limit}. Cohn, de Laat and Salmon presented better upper bounds for dimensions $3\leq n\leq 16$ except for $n=8$ \cite{SDP1}.
Besides, de Laat, Filho and Vallentin  obtained new bounds for packing by extending the linear programming method to semidefinite programming method  \cite{SDP}.

Based on these considerations, we cannot help but wonder: can we extend the method developed by Cohn and Elkies to include more auxiliary functions in the linear programming approach? In this paper, we provide a positive answer to this question and  {the contribution of this paper is twofold:
(1) new linear programming framework. Theorem \ref{key1} introduces a novel linear programming framework based on $m\Lambda$-periodic auxiliary functions $\{g_m\}_{m\in \mathbb{N}^{+}}$, where $\Lambda$ is a specified full-rank lattice in $\mathbb{R}^n$. This framework is broader and more flexible than the auxiliary functions derived from the method of Cohn and Elkies, as detailed in Theorem \ref{true_set};
(2) specific constructs for optimal packing.  We construct specific functions $\{g_m\}_{m\in \mathbb{N}^{+}}$ for $n=1$ $g_2$  for  $n=2$ and $g_1$ for $n=3$ such that $\frac{g_m(\boldsymbol{0})}{\widehat{g}_m(\boldsymbol{0})}$ equal to  the corresponding optimal packing bounds, as demonstrated in Proposition \ref{thm: one_dim_auxiliary}, Proposition \ref{thm: g_22}, and Proposition \ref{prop: 3dim}, respectively. While the auxiliary function in the $1$-dimensional case may be considered trivial in some respects, it plays a crucial role in facilitating the construction of $2$-dimensional and $3$-dimensional examples.
Given that the method of Cohn and Elkies fails to achieve optimal packing density in $3$-dimensional case, and the construction of effective auxiliary functions in $2$-dimensional case remains an open problem, our examples provide valuable insights into addressing  sphere packing problem for dimensions $n=2$, $n=3$, and potentially even higher dimensions.
}

\section{Preliminary Tools}\label{pre}
In this section, we introduce some definitions and notations. First of all, we introduce the definitions respect to \textit{lattice}. For any fixed  dimension $n\in \mathbb{N}^{+}$, the lattice generated by $n$ linearly independent vectors $\boldsymbol{v}_{1},\cdots, \boldsymbol{v}_{n}\in \mathbb{R}^n$ is defined as
\[
\Lambda:=\mathcal{L}(\boldsymbol{v}_1,\ldots,\boldsymbol{v}_{n})=\left\{\sum_{i=1}^{n}k_i\boldsymbol{v}_i\ :\ k_i\in \mathbb{Z}\ \text{for}\ i=1,\ldots,n \right\}.
\]
Since the elements in $\{\boldsymbol{v}_k\}_{k=1}^n$ are linearly independent, the lattice $\Lambda$ is also called a \textit{full-rank lattice}.
For any $m\in \mathbb{N}^{+}$, take  
\[
\Lambda_m=m\Lambda=\left\{\sum_{i=1}^{n}k_i\boldsymbol{v}_i\ :\ k_i\in m\mathbb{Z}\ \text{for}\ i=1,\ldots,n \right\}.
\]The fundamental parallelepiped $\mathcal{P}(\Lambda_m)$ corresponding to  $\Lambda_m$ is defined as
\[
\mathcal{P}(\Lambda_m)=\left\{\sum_{i=1}^n x_i\boldsymbol{v}_i\ :\ 0\leq x_1,\ldots,x_{n}<m\right\}.
\]
Denote the determinant $|\Lambda_m|$ as the volume of $\mathcal{P}(\Lambda_m)$. The dual lattice $\Lambda_m^*$ of $\Lambda_m$ is defined as
\[
\Lambda_m^*=\{\boldsymbol{y}\ :\ \langle \boldsymbol{x},\boldsymbol{y}\rangle\in \mathbb{Z}\ \text{for}\ \text{all}\ {\boldsymbol{x}\in \Lambda_m}\}.
\]
 {
Furthermore, given a lattice $\Lambda_m\subset \mathbb{R}^n$, the \textit{quotient space} $\mathbb{R}^n/{\Lambda_m}$  and the \textit{quotient norm} $\|\cdot\|_{\mathbb{R}^n/{\Lambda_m}}$ are respectively  defined as
\[
{\mathbb{R}^n}/{\Lambda_m}:=\{[\boldsymbol{x}]\ :\ \boldsymbol{x}\in \mathbb{R}^n\} \qquad \text{and}\qquad \|[\boldsymbol{x}]\|_{\mathbb{R}^n/{\Lambda_m}}:=\inf\{\|\boldsymbol{u}\|_2\ :\ \boldsymbol{u}\in \boldsymbol{x}+{\Lambda_m}\},
\]
where $
[\boldsymbol{x}]=\{\boldsymbol{x}+\boldsymbol{v}\ :\ \boldsymbol{v}\in {\Lambda_m}\}.
$
For any $\boldsymbol{x}_1,\boldsymbol{x}_2\in \mathbb{R}^n$, $\boldsymbol{x}_1\sim \boldsymbol{x}_2$ if and only if $\boldsymbol{x}_1-\boldsymbol{x}_2\in \Lambda_m$. }

Now we establish the basic terminology under sphere packing problem. The packing density refers to the proportion of space that is filled by the spheres. One can prove that periodic packing is arbitrarily close to the highest packing density density \cite[Appendix A]{Cohn}. So it is enough to consider any periodic packing $\mathcal{P}$. Assume that $\mathcal{P}$ is a periodic sphere packing in $\mathbb{R}^{n}$ with disjoint union set $\cup_{\boldsymbol{x}\in C}B^n(\boldsymbol{x})$, where $B^n(\boldsymbol{x})$ is an open sphere of radius-$1$ centered at $\boldsymbol{x}\in \mathbb{R}^n$ and $C$ is the center point set. Then the density $\Delta_{\mathcal{P}}$ of $\mathcal{P}$ is
\[
\Delta_{\mathcal{P}}=\underset{m\rightarrow \infty}{{\lim}}\frac{\text{vol}(\mathcal{P}(\Lambda_m)\cap \mathcal{P})}{\text{vol}(\mathcal{P}(\Lambda_m))}.
\]
If we want to discuss the sphere packing problem  under center-points  instead of spheres, it is more convenient to consider the center density $\delta_{\mathcal{P}}$ as
\[
\delta_{\mathcal{P}}=\underset{m\rightarrow \infty}{{\lim}}\frac{|\mathcal{P}(\Lambda_m)\cap C|}{\text{vol}(\mathcal{P}(\Lambda_m))},
\]
where $C$ is the center point set defined above. 
Then the \textit{sphere packing density} $\Delta_n$ and the \textit{center density} $\delta_n$  in $\mathbb{R}^n$ are simply the highest density, that is,
\begin{equation}\label{center}
\Delta_n=\sup_{\mathcal{P}}\Delta_{\mathcal{P}}\qquad \text{and}\qquad \delta_n=\sup_{\mathcal{P}}\delta_{\mathcal{P}}\qquad \text{with}\quad \Delta_n=\frac{\pi^{n/2}}{\Gamma(n/2+1)}\delta_n.
\end{equation}

The Cohn-Elkies linear programming bound heavily relies on some good auxiliary functions.
A function $f: \mathbb{R}^n\rightarrow \mathbb{R}$ is called \textit{admissible} if there exists a constant $\delta>0$ such that $|f(\boldsymbol{x})|$ and $|\widehat{f}(\boldsymbol{x})|$ are bounded above
by a constant times $(1+|\boldsymbol{x}|)^{-n-\delta}$ \cite{Cohn1}.
When linearly independent vectors $\boldsymbol{v}_1,\ldots,\boldsymbol{v}_n\in \mathbb{R}^n$ are given, we can define  the set of auxiliary functions \text{SA},  the set of $\Lambda_m$-periodic auxiliary functions {$\text{SA}_{m}$}, and the set of sequences of some $\Lambda_m$-periodic auxiliary functions \text{SSA} as below.
\begin{defi}
\label{cohn function}
\begin{enumerate}[{1)}]
\item the set of auxiliary functions \text{SA}: $f\in \text{SA}$, if  $f: \mathbb{R}^n\rightarrow \mathbb{R}$ is  admissible and satisfies:

 (i) $f(\boldsymbol{x})\leq 0$ for $\|\boldsymbol{x}\|_2\geq 1$;

 (ii) $\widehat{f}(\boldsymbol{t})\geq 0$ for all $\boldsymbol{t}\in \mathbb{R}^n$;

  (iii) $\widehat{f}(\boldsymbol{0})>0$.

 \item  the set of $\Lambda_m$-periodic auxiliary functions {$\text{SA}_{m}$}: $g_m\in \text{SA}_{m}$, if  $g_m: \mathbb{R}^n\rightarrow \mathbb{R}$ satisfies:

  (i) $g_m(\boldsymbol{x})=g_m(\boldsymbol{x}+\boldsymbol{v})$ for any $\boldsymbol{v}\in \Lambda_m$ and $\boldsymbol{x}\in \mathbb{R}^n$;

(ii) For any $\boldsymbol{x}\in \mathbb{R}^n$, $
g_{m}(\boldsymbol{x})= \sum_{\boldsymbol{t}\in \Lambda_m^*}c_{\boldsymbol{t}}\cos(2\pi\boldsymbol{t} \bm{\cdot} \boldsymbol{x})$,
where $c_{\boldsymbol{0}}>0$, $c_{\boldsymbol{t}}\geq 0$ and { $\sum_{\boldsymbol{t}\in \Lambda_{m}^*} c_{\boldsymbol{t}}<\infty$};

(iii)   $g_{m}(\boldsymbol{x})\leq0$ for $\|[\boldsymbol{x}]\|_{\mathbb{R}^{n}/\Lambda_{m}}\geq1$.

 \item the set of  sequences of some $\Lambda_m$-periodic auxiliary functions \text{SSA}: $\{g_m\}_{m\in \mathbb{N}^{+}}\in \text{SSA}$,
  if $g_m\in \text{SA}_{m}$, for  any $m\in \mathbb{N}^{+}$.
  \end{enumerate}
\end{defi}

{The definition  of SA is almost the same as the conditions in \cite[Theorem 3.1]{Cohn}, while $\widehat{f}(\boldsymbol{0})>0$ does not directly appear in \cite{Cohn}. However, it is only for the convenience of the calculation to make the upper bound of the sphere packing not vacuous, and such kind of restriction does not drop any useful auxiliary functions.  Moreover, we define three constants related to \text{SA}, $\text{SA}_m$ and SSA. 
\begin{defi}\label{param}
For any fixed  $f\in \text{SA}$, $g_m\in \text{SA}_m$ and $\{g_m\}_{m\in \mathbb{N}^{+}}\in \text{SSA}$, denote ${f}^{\#}$, $g_m^{\#}$ and $\{g_m\}_{m\in \mathbb{N}^{+}}^{\#}$ as
\[
{f}^{\#}:=\frac{f(\boldsymbol{0})}{\widehat{f}(\boldsymbol{0})},
\qquad g_m^{\#}:=\frac{g_m(\boldsymbol{0})}{\widehat{g}_m(\boldsymbol{0})}
\qquad \text{and}\qquad \{g_m\}_{m\in \mathbb{N}^{+}}^{\#}:=\underset{m\rightarrow \infty}{\underline{\mathrm{lim}}}\frac{g_m(\boldsymbol{0})}{\widehat{g}_m(\boldsymbol{0})}.
\]
Here we take  $\widehat{f}(\boldsymbol{0})=\int_{\mathbb{R}^n}f(\boldsymbol{x})d\boldsymbol{x}$ when  $f\in \text{SA}$, and  $\widehat{g}_m(\boldsymbol{0})=\int_{\mathcal{P}(\Lambda_m)}g_m(\boldsymbol{x})d\boldsymbol{x}$ when $g_m\in \text{SA}_{m}$, respectively.
\end{defi}
\begin{rem}\label{g remark}
For any $g_m\in \text{SA}_m$, since $g_m=\sum_{\boldsymbol{t}\in \Lambda_m^*}c_{\boldsymbol{t}}\cos(2\pi \boldsymbol{t}\bm\cdot \boldsymbol{x})$, then we can directly have 
\[
\widehat{g}_m(\boldsymbol{0})=c_{\boldsymbol{0}}|\Lambda_m|.
\]
It leads to 
\[
g_m^\#=\frac{g_m(\boldsymbol{0})}{c_{\boldsymbol{0}}|\Lambda_m|}=\frac{g_m(\boldsymbol{0})}{m^n c_{\boldsymbol{0}}|\Lambda|}. 
\]
\end{rem}
\section{Main Results}
In this section, we present our main results.  {In Section \ref{pre}, we introduce the sequence of auxiliary functions SSA, which plays a fundamental role in our linear programming framework. }
First of all, we will show that for any $f\in \text{SA}$, there exists a corresponding sequence $\{f_m\}_{m\in \mathbb{N}^{+}}\in \text{SSA}$ such that $\{f_m\}_{m\in \mathbb{N}^{+}}^{\#}= {f}^{\#}$. 
\begin{thm}\label{Tf}
Denote the transform $T$ on a given function $f$ as $Tf=\{f_1,\cdots,f_m,\ldots\}$, where $f_m$ is defined as 
  \[f_m(\boldsymbol{x}):=T_mf(\boldsymbol{x}):=\sum_{\boldsymbol{t}\in \Lambda_m}f(\boldsymbol{x}+\boldsymbol{t}),\qquad \text{for}\ m=1,2,\ldots.\]
  Then for any fixed $f\in \text{SA}$, we have
  \begin{enumerate}[{(1)}]
  \item $Tf\in \text{SSA}$;
  \item $\{f_m\}_{m\in \mathbb{N}^{+}}^{\#}= {f}^{\#}$.
  \end{enumerate}
\end{thm}
\begin{proof}
The proof is provided in Section \ref{prof Tf}.
\end{proof}
 {Before we present the upper bound of center density $\delta_n$ according to SSA, } we restate the linear programming bound of  sphere packing problem in \cite[Theorem 3.1]{Cohn} by Theorem \ref{thm: cohn}.
\begin{thm}\cite{Cohn}\label{thm: cohn}
For any fixed dimension $n$, the center density $\delta_n$ satisfies
\[
\delta_n\leq \frac{1}{2^n}\underset{f\in \text{SA}}\inf f^{\#},
\]
where \text{SA} is the auxiliary set defined in Definition \ref{cohn function}.
\end{thm}

The following theorem shows that we can also get the upper bound of the center density $\delta_n$ based on SSA.
According to Theorem \ref{Tf} and Theorem \ref{key1} below, we can conclude that the auxiliary function set SSA is broader than SA with possibly tighter upper bound estimation.
\begin{thm}
\label{key1}
For any fixed dimension $n$, the center density $\delta_n$ satisfies
\[
\delta_n\leq \frac{1}{2^n}\underset{\{g_m\}_{m\in \mathbb{N}^{+}}\in \text{SSA}}\inf \{g_m\}_{m\in \mathbb{N}^{+}}^{\#},
\]
where SSA is the set of  sequences of some auxiliary functions defined in Definition \ref{cohn function}.
\end{thm}
\begin{proof}
The proof is provided  in Section \ref{proof1}.
\end{proof}

\begin{rem}\label{one-dimen-remark}
 {Consider $1$-dimensional auxiliary functions $g(\cdot)$ and $h(\cdot)$ introduced by Cohn and Elkies \cite{Cohn}, that is,
\[
g(x)=(1-|x|)\chi_{[-1,1]}(x) \qquad  \mbox{and} \qquad  h(x)=\frac{1}{(1-x^2)}\cdot\frac{\sin^2(\pi x)}{(\pi x)^2},
\]
 where $\chi_{[-1,1]}$ is the characteristic function on $[-1,1]$, they solve the sphere packing in dimension 1. We can see that these two useful auxiliary functions also have good performance in our framework. 
By the Poisson summation formula \cite[Equation (2.1)]{Cohn} and Theorem \ref{Tf}, we have
\[
g_{m}(x)=\frac{1}{|\Lambda_m|}\sum_{t\in \Lambda_m^*}c_{t,1}\cos(2\pi tx)\quad \text{and}\quad  h_{m}(x)=\frac{1}{|\Lambda_m|}\sum_{t\in \Lambda_m^*}c_{t,2}\cos(2\pi tx),
\]
where $c_{0,1}=c_{0,2}=1$ and $g_{m}(0)=h_{m}(0)=\widehat{g_{m}}(0)=\widehat{h_{m}}(0)=1$ for every $m\in N^+$.
Then, by Definition \ref{param}, $\{g_m\}_{m\in \mathbb{N}^{+}}^{\#}=\{f_m\}_{m\in \mathbb{N}^{+}}^{\#}=1$.
Hence, according to Theorem \ref{key1}, $\delta_1\leq \frac{1}{2}$. }
\end{rem}

 {Based on Remark \ref{one-dimen-remark} and Theorem \ref{Tf}, readers may wonder whether it suffices to apply Cohn and Elkies' functions in SA to construct the sequences of auxiliary functions in SSA. However, the auxiliary functions in SA may not achieve the optimal packing bounds in high dimensions. As the functions in SSA is more flexible than those in SA, we need to develop new sequences of auxiliary functions in SSA that do not originate from SA and still meet the optimal bounds.
Proposition \ref{thm: one_dim_auxiliary} introduces a new sequence of auxiliary functions, denoted as $\{g_m\}^{\#}_{m\in \mathbb{N}^{+}} \in \text{SSA}$, which is efficient in $1$-dimensional case.  Furthermore, the auxiliary functions presented in Proposition \ref{thm: one_dim_auxiliary} cannot be constructed from any functions in SA. More details can be found in Proposition \ref{true_set}. Proposition \ref{thm: one_dim_auxiliary} provides insights into the $2$-dimensional, $3$-dimensional, and even higher-dimensional cases.}
\begin{prop}\label{thm: one_dim_auxiliary}
For any fixed $m\in\{3,4,\ldots\}$, define the function $h_{m}$ as
\[
h_{m}(z)=\frac{(z^{m}-1)^{2}}{(z-1)^{2}\cdot z^{m-2}\cdot(z-\zeta_{m})(z-\overline{\zeta_{m}})},
\]
where $\zeta_{m}=e^{\mathrm{i}\frac{2\pi}{m}}$. Define the continuous function $g_m: \mathbb{R}\rightarrow \mathbb{C}$ such that
$g_{m}(x)=h_{m}(e^{\mathrm{i}\frac{2\pi x}{m}})$, for $x\notin \mathbb{Z}$.
Then $g_{m}\in SA_{m}$, for $m=3,4,\ldots$, and 
$
\{g_m\}_{m\in \mathbb{N}^{+}}^{\#}=1.
$
\end{prop}
\begin{proof}
The proof  is provided  in Section \ref{proofs}.
\end{proof}

{
The sequence of the auxiliary functions in Proposition \ref{thm: one_dim_auxiliary}  solves the $1$-dimensional sphere packing problem.
Furthermore, Proposition \ref{true_set} shows that  such kind of auxiliary function sequence cannot be generated by $Tf$ for any $f\in \text{SA}$.
 The parameters $\{\alpha_m\}_{m\in \mathbb{N}^{+}}$ in Proposition \ref{true_set}
is added to increase the flexibility of scaling when different $m$ is chosen and does not influence the center density estimation in Theorem \ref{key1}. }
\begin{prop}\label{true_set}
Assume that the auxiliary function $g_m:\mathbb{R}\rightarrow \mathbb{R}$ is defined in  Proposition \ref{thm: one_dim_auxiliary} and the transform $T$ is defined in Theorem \ref{Tf}. For any sequence of positive constants $\{\alpha_m\}_{m\in \mathbb{N}^{+}}$, there does not exist any function $f\in \text{SA}$, such that 
\[
Tf=\{\alpha_m g_m\}_{m\in \mathbb{N}^{+}}.
\] 
\end{prop}
\begin{proof}
The proof is provided in Section \ref{sec: true_set}.
\end{proof}
\begin{rem}
Combing Theorem \ref{Tf} with Proposition \ref{true_set},  we can see that the auxiliary functions in SSA is broader than those in SA.
Hence, if we can construct a good auxiliary function sequence belongs to SSA, we can possibly get better upper bounds on center density based on Theorem \ref{key1}.
\end{rem}
 {Now we turn our attention to the construction of effective auxiliary functions in $2$-dimensional case. In \cite{Cohn}, Cohn and Elkies demonstrated  that the linear programming bound established by Cohn and Elkies is sharp through numerical experiments. They aimed to construct good auxiliary functions within SA. However, the explicit form of such kind of auxiliary function is still not found. 
Proposition \ref{thm: g_22} introduces the explicit form of an auxiliary function $g_2 \in \text{SA}_{2}$ that matches the optimal packing bound. According to Theorem \ref{Tf}, there exists a connection between the functions in $\text{SA}_m$ and SA. If an efficient auxiliary function in SA could be identified to address the $2$-dimensional sphere packing problem, it would allow for the derivation of effective auxiliary functions in $\text{SA}_m$ through the transformation $T_m$.
Although we have not yet formulated the functions in $\text{SA}_m$ for general $m$, the construction presented in Proposition \ref{thm: g_22} offers suggestive directions for the development of auxiliary functions in SA.} 

\begin{prop}
\label{thm: g_22}
Take
\[
\Lambda_{2}=\{k_1\boldsymbol{v}_1+k_2\boldsymbol{v}_2\ :\ k_1,k_2\in 2\mathbb{Z}\},
\]
where  $\boldsymbol{v}_{1}=[1,0]^{\mathsf{T}}$and $\boldsymbol{v}_{2}=[\frac{1}{2},\frac{\sqrt{3}}{2}]^{\mathsf{T}}$.
Take ${c}_{0}=c_{1}=c_{2}=c_{3}=1$, and
\[
\boldsymbol{a}_{0}=[0,0]^{\mathsf{T}},\quad\boldsymbol{a}_{1}=[\pi,-\frac{1}{\sqrt{3}}\pi]^{\mathsf{T}},\quad\boldsymbol{a}_{2}=[0,\frac{2}{\sqrt{3}}\pi]^{\mathsf{T}},\quad\text{and}\quad\boldsymbol{a}_{3}=[\pi,\frac{1}{\sqrt{3}}\pi]^{\mathsf{T}}.
\]
Denote $g_{2}:\mathbb{R}^{2}\rightarrow\mathbb{R}$
as

\[
g_{2}(\boldsymbol{x})=\sum_{k=0}^{3}c_{k}\cos(\boldsymbol{a}_{k}\bm{\cdot} \boldsymbol{x}).
\]
 {Then $g_{2}\in \text{SA}_2$ and $g_2^{\#}=\frac{2}{\sqrt{3}}$. }
\end{prop}
\begin{proof}
The proof is provided in Section \ref{sec: Proof of Proposition 8}.
\end{proof}
\begin{rem}

If we can construct auxiliary functions $g_m\in \text{SA}_m$ such that $g_m^{\#}=\frac{2}{\sqrt{3}}$ for any integer $m\geq 3$,
then we have $\{g_m\}_{m\in \mathbb{N}^{+}}^{\#}=\frac{2}{\sqrt{3}}$ and $\delta_2$ is upper bounded by $\frac{1}{2\sqrt{3}}$ according to Theorem \ref{key1}.
Moreover, specific radius-$\frac{1}{2}$ sphere packing at $\mathbb{R}^2/\Lambda_2$ can be constructed in Figure \ref{fig1}, where the same shallow pattern shares the same sphere on $\mathbb{R}^2/\Lambda_2$.
\begin{figure}[htbp]
	\centering
	{\includegraphics[width=0.4\textwidth]{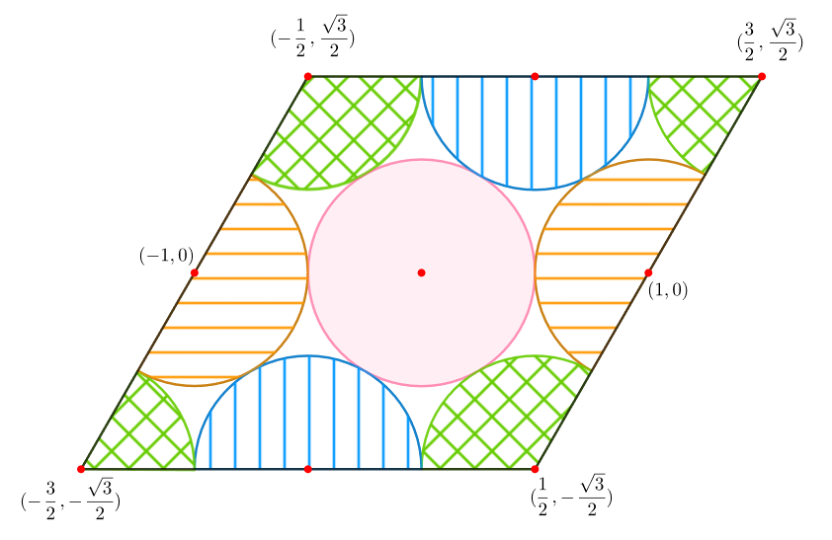}}
	\caption{Specific radius-$\frac{1}{2}$ sphere packing in $2$-dimensional case}
	\label{fig1}
\end{figure}
Then the lower bound of the center density $\delta_2$ should be $\delta_2\geq\frac{4}{2^{2}\cdot2\sqrt{3}}=\frac{1}{2\sqrt{3}}.$
Therefore, $\delta_2=\frac{1}{2\sqrt{3}}$, if  good auxiliary functions can be constructed for $m\geq 3$, and the conjecture by Cohn and Elkies \cite{Cohn} will be proven. Here we leave it for future works.
\end{rem}
 {
Last but not least, we consider the auxiliary functions in $3$-dimensional case. It is important to highlight that Cohn-Elkies' linear programming method does not achieve optimal sphere packing in this context \cite{limit,Li}.  However, drawing on Theorem \ref{Tf} and Proposition \ref{true_set}, we find that the class of sequences of $\Lambda_m$-periodic auxiliary functions in SSA is more extensive than that of the auxiliary functions proposed by Cohn and Elkies. Based on Theorem \ref{Tf} and  \cite[Theorem 6.1]{Li}, we have $T_1f\geq 0.18398089\times8\geq 1.4718>\sqrt{2}$ for any $f\in \text{SA}$.  In Proposition \ref{prop: 3dim}, we can construct an efficient auxiliary function $g \in \text{SA}_1$. Notably, $g^{\#}=\sqrt{2}$, which  equals to the packing bound through a face-centered cubic (f.c.c.) lattice arrangement. This indicates a promising potential for achieving the optimal bound within our linear programming framework.}
\begin{prop}\label{prop: 3dim}
Take \[\Lambda=\{k_1\boldsymbol{v}_1+k_2\boldsymbol{v}_2+k_3\boldsymbol{v}_3\ :\ k_1,k_2,k_3\in \mathbb{Z}\},\] where
$\boldsymbol{v}_1=[\sqrt{2},0,0]^{\mathsf{T}}$, $\boldsymbol{v}_2=[0,\sqrt{2},0]^{\mathsf{T}}$ and $\boldsymbol{v}_3=[0,0,\sqrt{2}]^{\mathsf{T}}$. Take $c_0=c_1=c_2=c_3=1$, and 
\[
\boldsymbol{a}_0=[0,0,0]^{\mathsf{T}},\qquad \boldsymbol{a}_1=[\sqrt{2}\pi,0,0]^{\mathsf{T}},\qquad \boldsymbol{a}_2=[0, \sqrt{2}\pi,0]^{\mathsf{T}}\qquad \text{and}\qquad \boldsymbol{a}_3=[0,0,\sqrt{2}\pi]^{\mathsf{T}}.
\]
Denote $g: \mathbb{R}^3\rightarrow \mathbb{R}$ as 
\[
g(\boldsymbol{x})=\sum_{k=0}^3 c_k\cos(\boldsymbol{a}_k\bm{\cdot}\boldsymbol{x}).
\]
 {Then $g\in \text{SA}_{1}$ and  $g^{\#}=\sqrt{2}$. }
\end{prop}
\begin{proof}
The proof is provided in Section \ref{dim_3: sec}. 
\end{proof}
\begin{rem}
Consider $13$ points in Figure \ref{fig2}, that is, $\boldsymbol{O}=(0,0,0)$, $\{\boldsymbol{P}_1,\boldsymbol{P}_2,\boldsymbol{P}_3,\boldsymbol{P}_4\}=(\pm \frac{\sqrt{2}}{2},\pm\frac{\sqrt{2}}{2},0)$,  $\{\boldsymbol{P}_5,\boldsymbol{P}_6,\boldsymbol{P}_7,\boldsymbol{P}_8\}=(\pm \frac{\sqrt{2}}{2},0,\pm\frac{\sqrt{2}}{2})$ and  $\{\boldsymbol{P}_9,\boldsymbol{P}_{10},\boldsymbol{P}_{11},\boldsymbol{P}_{12}\}=(0,\pm \frac{\sqrt{2}}{2},\pm\frac{\sqrt{2}}{2})$. Assume that $\boldsymbol{x}\sim \boldsymbol{y}$, when $\boldsymbol{x}-\boldsymbol{y}\in \Lambda$, where $\Lambda$ is defined in Proposition \ref{prop: 3dim}. Then the quotient space will be $\mathbb{R}^3/\Lambda$ and we can find that 
\[
\boldsymbol{P}_1\sim \boldsymbol{P}_2\sim \boldsymbol{P}_3\sim \boldsymbol{P}_4,\quad \boldsymbol{P}_5\sim \boldsymbol{P}_6\sim \boldsymbol{P}_7\sim \boldsymbol{P}_8,\quad \text{and}\quad \boldsymbol{P}_9\sim \boldsymbol{P}_{10}\sim \boldsymbol{P}_{11}\sim \boldsymbol{P}_{12}.
\]
Since $\|\boldsymbol{O}\boldsymbol{P}_i\|_2=1$, for any $i=1,\ldots,12$, and $\|\boldsymbol{P}_{i}\boldsymbol{P}_{j}\|_2\geq 1$, for  any $1\leq i<j\leq 12$, we can take $[\boldsymbol{O}],[\boldsymbol{P}_1],[\boldsymbol{P}_5],[\boldsymbol{P}_9]$ as the center points of the radius-$\frac{1}{2}$ sphere packing in $3$-dimensional case. For a geometric perspective, if we consider the intersection of $B(\boldsymbol{P},\frac{1}{2})$ with the cube $\left[-\frac{\sqrt{2}}{2},\frac{\sqrt{2}}{2}\right]^3$ (where $B(\boldsymbol{P},\frac{1}{2})$ is a radius-$\frac{1}{2}$ sphere centered at $\boldsymbol{P}$), we have:
\[
B(\boldsymbol{O},\frac{1}{2})\cap \left[-\frac{\sqrt{2}}{2},\frac{\sqrt{2}}{2}\right]^3=B(\boldsymbol{O},\frac{1}{2}),\quad and\quad  B(\boldsymbol{P}_i,\frac{1}{2})\cap \left[-\frac{\sqrt{2}}{2},\frac{\sqrt{2}}{2}\right]^3=\frac{1}{4}B(\boldsymbol{P}_i,\frac{1}{2}),
\]
for $i=1,\ldots, 12$, where $\frac{1}{4}B(\boldsymbol{P}_i,\frac{1}{2})$ represents a quarter sphere with radius of $1/2$ and the center $\boldsymbol{P}_i$. Therefore, there are $1+12\times \frac{1}{4}=4$ radius-$\frac{1}{2}$ spheres in the cube $\left[-\frac{\sqrt{2}}{2},\frac{\sqrt{2}}{2}\right]^3$, and the center density in the $3$-dimensional case should be 
\begin{equation}\label{eqn: dela_3_lower}
\delta_3\geq \frac{1}{2^3} \cdot \frac{4}{2\sqrt{2}}=\frac{1}{4\sqrt{2}}.
\end{equation}
As we can see, this type of periodic packing corresponds to  the f.c.c. lattice arrangement discussed  in Kepler's conjecture \cite{nature}. 
\begin{figure}[htbp]
	\centering
	{\includegraphics[width=0.4\textwidth]{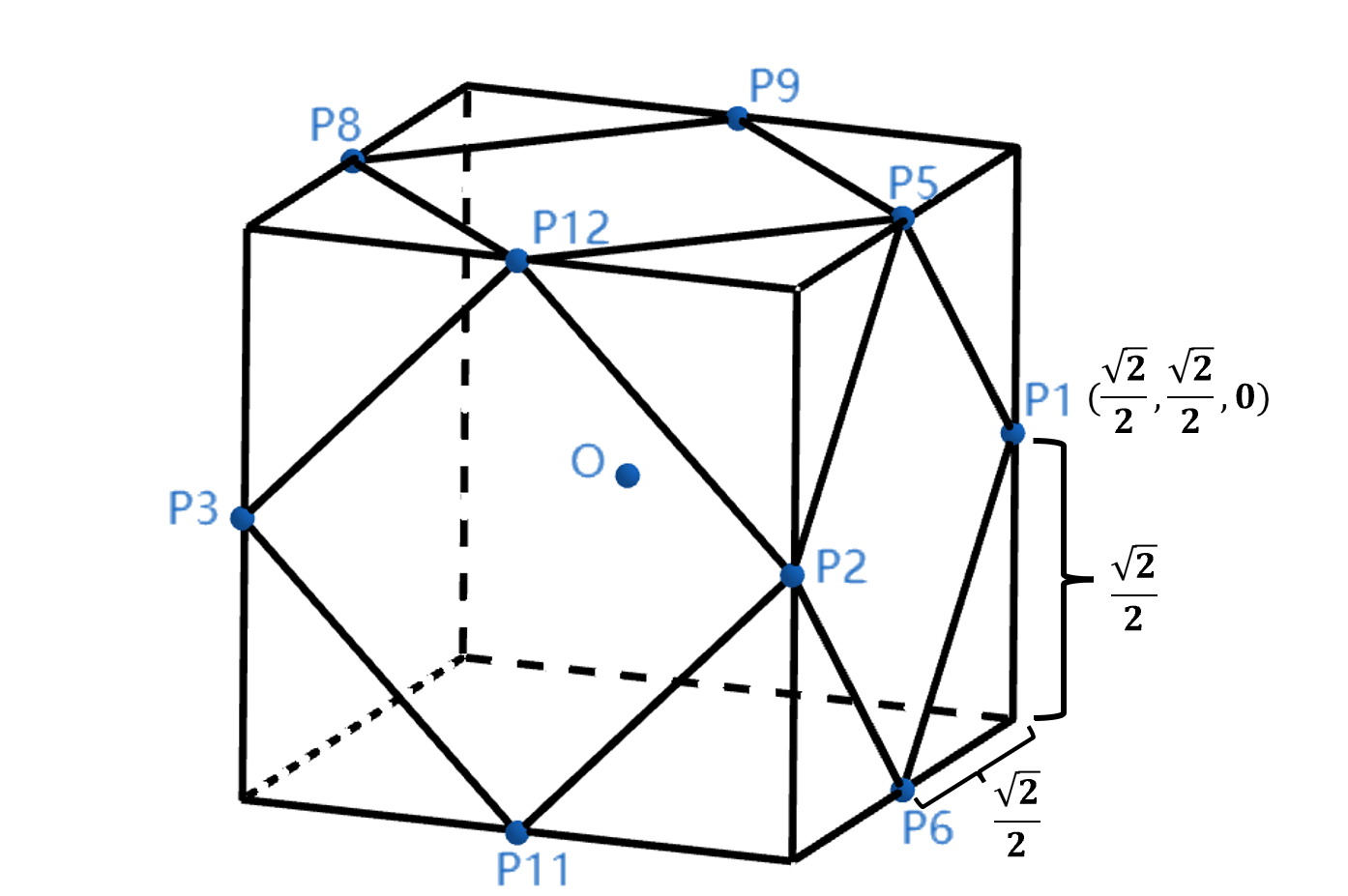}}
	\caption{Specific radius-$\frac{1}{2}$ sphere packing in $3$-dimensional case}
	\label{fig2}
\end{figure}

On the other hand, we also have  $\frac{1}{2^3}\cdot g^{\#}=\frac{1}{4\sqrt{2}}$.
It meets the lower bound of $\delta_3$ in (\ref{eqn: dela_3_lower}). It implies that we have already constructed a good auxiliary function in $\text{SA}_1$. 
\end{rem}

Proposition \ref{thm: one_dim_auxiliary}, Proposition \ref{thm: g_22} and Proposition \ref{prop: 3dim} offer valuable insights into the construction of effective auxiliary functions $\{g_m\}_{m\in \mathbb{N}^{+}}\in \text{SSA}$, particularly in two and three dimensional cases. Therefore,  based on Theorem \ref{key1},  there is a promising outlook for solving the sphere packing problem in  two and three dimensions, and potentially for  $n=4,\ldots,9$, under optimal lattice arrangements.

\section{Proofs}
\subsection{Proof of Theorem \ref{Tf}}\label{prof Tf}
\begin{proof}[Proof of Theorem \ref{Tf}]
(1) In order to prove $Tf\in \text{SSA}$, it is enough to prove that $f_m\in \text{SA}_{m}$, for any $m\in \mathbb{N}^{+}$.
Now we should check the conditions (i), (ii) and (iii) for $\text{SSA}_{m}$ in Definition \ref{cohn function} one by one:

Based on the definition of $f_m$, the periodic condition in (i) can be directly got.

According to the Poisson summation formula, we have
\begin{equation}\label{possion}
f_m(\boldsymbol{x})=\sum_{\boldsymbol{t}\in \Lambda_m}f(\boldsymbol{x}+\boldsymbol{t})=\frac{1}{|\Lambda_m|}\sum_{\boldsymbol{t}\in \Lambda_m^*}e^{-2\pi \mathrm{i}\boldsymbol{x}\bm{\cdot}\boldsymbol{t}}\widehat{f}(\boldsymbol{t})=\mathscr{R}e\left(\frac{1}{|\Lambda_m|}\sum_{\boldsymbol{t}\in \Lambda_m^*}e^{-2\pi \mathrm{i}\boldsymbol{x}\bm{\cdot}\boldsymbol{t}}\widehat{f}(\boldsymbol{t})\right)=\sum_{\boldsymbol{t}\in\Lambda_m^*}c_{\boldsymbol{t}}\cos(2\pi \boldsymbol{t}\bm{\cdot}\boldsymbol{x}),
\end{equation}
where $c_{\boldsymbol{t}}=\frac{1}{|\Lambda_m|}\widehat{f}(\boldsymbol{t})$ for any $\boldsymbol{t}\in \Lambda_m^*$.
Since $\widehat{f}(\boldsymbol{t})\geq 0$ for any $\boldsymbol{t}\in \Lambda_m^*$ and $\widehat{f}(\boldsymbol{0})> 0$, we have  $c_{\boldsymbol{t}}\geq 0$ for any $\boldsymbol{t}\in \Lambda_m^*$ and $c_{\boldsymbol{0}}> 0$. Besides, since $f\in \text{SA}$, $f$ is admissible and  both sides of (\ref{possion}) converge absolutely.
Taking $\boldsymbol{x}=\boldsymbol{0}$ in (\ref{possion}), we have $\sum_{\boldsymbol{t}\in\Lambda_m^*}c_{\boldsymbol{t}}<\infty$.  Thus (ii) holds.

Moreover, if $\|[\boldsymbol{x}]\|_{\mathbb{R}^n/\Lambda_m}\geq 1$, then for any $\boldsymbol{t}\in \Lambda_m$, we have $\|\boldsymbol{x}+\boldsymbol{t}\|_2\geq 1$. Since $f(\boldsymbol{x})\leq 0$ for any $\|\boldsymbol{x}\|_2\geq 1$, when $\|[\boldsymbol{x}]\|_{\mathbb{R}^n/\Lambda_m}\geq 1$, we have
\[
f_m(\boldsymbol{x})=\sum_{\boldsymbol{t}\in \Lambda_m}f(\boldsymbol{x}+\boldsymbol{t})\leq 0.
\]
Thus we can get (iii).

(2) Since $f$ is admissible, then
\[\underset{m\rightarrow \infty}{\lim}f_m(\boldsymbol{0})=f(\boldsymbol{0})+\underset{m\rightarrow \infty}{\lim}\sum_{\boldsymbol{t}\in \Lambda_m\backslash  \{\boldsymbol{0}\}}f(\boldsymbol{t})=f(\boldsymbol{0}).\]

On the other hand, we can directly have
\[
\widehat{f}_m(\boldsymbol{0})=\int_{\mathcal{P}(\Lambda_m)}f_m(\boldsymbol{x})d\boldsymbol{x}=\int_{\mathcal{P}(\Lambda_m)}\sum_{\boldsymbol{t}\in \Lambda_m}f(\boldsymbol{x}+\boldsymbol{t})d\boldsymbol{x}=\int_{\mathbb{R}^n}f(\boldsymbol{x})d\boldsymbol{x}=\widehat{f}(\boldsymbol{0}).
\]
Thus
\[
\{f_m\}_{m\in \mathbb{N}^{+}}^{\#}=\underset{m\rightarrow \infty}{\underline{\mathrm{lim}}}\frac{f_m(\boldsymbol{0})}{\widehat{f}_m(\boldsymbol{0})}=\frac{f(\boldsymbol{0})}{\widehat{f}(\boldsymbol{0})}={f}^{\#},
\]
which leads to the conclusion.
\end{proof}

\subsection{Proof of Theorem \ref{key1}}
\label{proof1}
\begin{proof}[Proof of Theorem \ref{key1}]
Periodic packing is arbitrarily close to the greatest packing
density \cite[Appendix A]{Cohn}. Therefore, it is enough to consider periodic packing under lattice $\mathbb{R}^n/\Lambda_m$ with $m$ large enough. Assume
that there exist $N$ nodes $\boldsymbol{x}_{1},\cdots,\boldsymbol{x}_{N}\in \mathbb{R}^n/\Lambda_m$
such that $\|[\boldsymbol{x}_i]-[\boldsymbol{x}_j]\|_{\mathbb{R}^{n}/{\Lambda_{m}}}\geq1$ for any $i\neq j$. Then $\{\boldsymbol{x}_{j}\}_{j=1}^{N}$ can be
considered as the centers of spheres in $\mathbb{R}^n/{\Lambda_{m}}$ with radius-$1/2$. Therefore,
the center density respect to $\Lambda_m$ is upper bounded by
\[
\frac{N}{2^{n}|\Lambda_m|}.
\]
For any fixed $\{g_m\}_{m\in \mathbb{N}^{+}}\in \text{SSA}$, we have  $g_m\in \text{SA}_m$, $m=1,2,\ldots.$ By the definition of $\text{SA}_m$, there exist $c_{\boldsymbol{t}}$, $\boldsymbol{t}\in\Lambda_m^*$, such that $c_{\boldsymbol{0}}> 0$, $c_{\boldsymbol{t}}\geq  0$, { $\sum_{\boldsymbol{t}\in \Lambda_{m}^*} c_{\boldsymbol{t}}<\infty$}, and  $
g_{m}(\boldsymbol{x})= \sum_{\boldsymbol{t}\in \Lambda_m^*}c_{\boldsymbol{t}}\cos(2\pi\boldsymbol{t} \bm{\cdot} \boldsymbol{x})$ for all $\boldsymbol{x}\in \mathbb{R}^n$. Denote
\[
S_{m}(\boldsymbol{t})=\sum_{1\leq i,j\leq N}g_{m}(\boldsymbol{t}+\boldsymbol{x}_{i}-\boldsymbol{x}_{j}).
\]
 Since $\|[\boldsymbol{x}_i]-[\boldsymbol{x}_j]\|_{\mathbb{R}^{n}/{\Lambda_{m}}}=\|[\boldsymbol{x}_i-\boldsymbol{x}_j]\|_{\mathbb{R}^{n}/{\Lambda_{m}}}\geq1$ for any
$i\neq j$, according to the condition (iii) for $\text{SA}_m$,
we have $g_{m}(\boldsymbol{x}_{i}-\boldsymbol{x}_{j})\leq0$. Therefore,
\begin{equation}
S_{m}(\boldsymbol{0})=\sum_{1\leq i,j\leq N}g_{m}(\boldsymbol{x}_{i}-\boldsymbol{x}_{j})=\sum_{i=j}g_{m}(\boldsymbol{x}_{i}-\boldsymbol{x}_{j})+\sum_{i\neq j}g_{m}(\boldsymbol{x}_{i}-\boldsymbol{x}_{j})\leq Ng_{m}(\boldsymbol{0}).\label{eq: one_dim_temp1-1}
\end{equation}
On the other hand, by the condition (ii) for $\text{SA}_m$, we also have
\begin{align}
S_{m}(\boldsymbol{0}) & = \sum_{1\leq i,j\leq N}\sum_{\boldsymbol{t}\in \Lambda_m^*}c_{\boldsymbol{t}}\cos(2\pi\boldsymbol{t} \bm{\cdot} (\boldsymbol{x}_i-\boldsymbol{x}_j))=\sum_{\boldsymbol{t}\in \Lambda_m^*}{c_{\boldsymbol{t}}}\sum_{1\leq i,j\leq N}\cos(2\pi{\boldsymbol{t}}\bm{\cdot}(\boldsymbol{x}_{i}-\boldsymbol{x}_{j}))\nonumber \\
 & =N^{2}c_{\boldsymbol{0}}+\sum_{\boldsymbol{t}\in \Lambda_m^*\backslash\{\boldsymbol{0}\}}c_{\boldsymbol{t}}\sum_{1\leq i,j\leq N}\cos(2\pi{\boldsymbol{t}}\bm{\cdot}(\boldsymbol{x}_{i}-\boldsymbol{x}_{j})).\label{eq: one_dimen_temp2-1}
\end{align}
For any fixed $\boldsymbol{t}\in \Lambda_{m}^*$, we have

\[
\sum_{1\leq i,j\leq N}e^{2\pi\mathrm{i}\boldsymbol{t}\bm{\cdot}(\boldsymbol{x}_{i}-\boldsymbol{x}_{j})}=\sum_{1\leq i\leq N}\sum_{1\leq j\leq N}e^{2\pi\mathrm{i}\boldsymbol{t}\bm{\cdot}\boldsymbol{x}_{i}}\cdot e^{-2\pi\mathrm{i}\boldsymbol{t}\bm{\cdot}\boldsymbol{x}_{j}}=\big|\sum_{1\leq i\leq N}e^{2\pi\mathrm{i} \boldsymbol{t}\bm{\cdot}\boldsymbol{x}_{i}}\big|^{2}.
\]
Then
\begin{equation}
\sum_{1\leq i,j\leq N}\cos(2\pi\boldsymbol{t}\bm{\cdot}(\boldsymbol{x}_{i}-\boldsymbol{x}_{j}))=\sum_{1\leq i,j\leq N}\mathscr{R}e(e^{2\pi\mathrm{i}\boldsymbol{t}\bm{\cdot}(\boldsymbol{x}_{i}-\boldsymbol{x}_{j})})\geq0.\label{eq: one_dimen_temp3-2}
\end{equation}
Plugging (\ref{eq: one_dimen_temp3-2}) into (\ref{eq: one_dimen_temp2-1}),
we can get that
\begin{equation}
S_{m}(\boldsymbol{0})\geq N^{2}c_{\boldsymbol{0}}.\label{eq: one_dimen_temp4-2}
\end{equation}
Based on (\ref{eq: one_dim_temp1-1}) and (\ref{eq: one_dimen_temp4-2}),
we have
\[
N\leq\frac{g_{m}(\boldsymbol{0})}{c_{\boldsymbol{0}}}=\frac{g_{m}(\boldsymbol{0})|\Lambda_m|}{\widehat{g}_m(\boldsymbol{0})},
\]
as $\widehat{g}_m(\boldsymbol{0})=|\Lambda_m|c_{\boldsymbol{0}}$. Thus for any fixed $\{g_m\}_{m\in \mathbb{N}^{+}}$, we have
\[
\delta_n\leq \underset{m\rightarrow\infty}{ \underline{\text{lim}}}\frac{N}{2^{n}|\Lambda_m|}=\underset{m\rightarrow\infty}{ \underline{\text{lim}}}\frac{g_{m}(\boldsymbol{0})}{\widehat{g}_m(\boldsymbol{0})}=\{g_m\}_{m\in \mathbb{N}^{+}}^{\#},
\]
which leads to the conclusion.
\end{proof}

\subsection{Proof of Proposition \ref{thm: one_dim_auxiliary}}
\label{proofs}
\begin{proof}[Proof of proposition \ref{thm: one_dim_auxiliary}]
The proof can separated by the following three steps:

\textbf{Step 1}: Establish that $g_m$ can be expressed as
\begin{equation}
g_{m}(x)=\sum_{k=0}^{m-2}c_{k}\cos\frac{2k\pi x}{m},\label{eq: g_m}
\end{equation}
where $c_{k}>0$, $k=0,\ldots,m-2$. 

 If $g_m$ can be expressed in (\ref{eq: g_m}), we can directly get that $g_m(x)=g_m(x+v)$ for any $v\in m\mathbb{Z}$ and the condition (i) and condition (ii) for $\text{SA}_m$ hold true. Now the thing left is to identify the equation form (\ref{eq: g_m}). Since $\overline{\zeta_{m}}=(\zeta_{m})^{m-1}$ and
$z^{m}-1=\prod_{k=0}^{m-1}(z-(\zeta_{m})^{k})$, then $h_{m}$ can be
represented as
\begin{equation}
h_{m}(z)=\frac{(z-\zeta_{m})(z-\overline{\zeta_{m}})\prod_{k=2}^{m-2}(z-(\zeta_{m})^{k})^{2}}{z^{m-2}}=\sum_{k=-m+2}^{m-2}\widetilde{c}_{k}z^{k},\label{eq: f_sum_representation}
\end{equation}
where the parameters $\widetilde{c}_{k}$ will be estimated later.

By some direct calculations, we have
\begin{align}
\frac{1}{(z-1)^{2}(z-\zeta_{m})(z-\overline{\zeta_{m}})}= & \frac{1}{(\zeta_{m}-1)(\overline{\zeta_{m}}-1)}\cdot\frac{1}{(z-1)^{2}}+\frac{\zeta_{m}+\overline{\zeta_{m}}-2}{(\zeta_{m}-1)^{2}(\overline{\zeta_{m}}-1)^{2}}\cdot\frac{1}{z-1}\label{eq: 1-z_temp}\\
 & +\frac{1}{(\zeta_{m}-\overline{\zeta_{m}})(\zeta_{m}-1)^{2}}\cdot\frac{1}{z-\zeta_{m}}-\frac{1}{(\zeta_{m}-\overline{\zeta_{m}})(\overline{\zeta_{m}}-1)^{2}}\cdot\frac{1}{z-\overline{\zeta_{m}}}.\nonumber
\end{align}
Furthermore, for any fixed $z_{0}\in\mathbb{C}$, when $z$ satisfies
$|z|<|z_{0}|$, we have
\begin{equation}
\frac{1}{z-z_{0}}=-\sum_{k=0}^{\infty}\frac{1}{(z_{0})^{k+1}}z^{k}\qquad\text{and}\qquad\frac{1}{(z-z_{0})^{2}}=\sum_{k=0}^{\infty}\frac{k+1}{(z_{0})^{k+2}}z^{k}.\label{eq: 1-z_expansion}
\end{equation}
Plugging (\ref{eq: 1-z_expansion}) into (\ref{eq: 1-z_temp}) with
$z_{0}=1$, $\zeta_{m}$ and $\overline{\zeta_{m}}$, it obtains that
\begin{align}
\small
 & \frac{1}{(z-1)^{2}(z-\zeta_{m})(z-\overline{\zeta_{m}})}\nonumber \\
= & \sum_{k=0}^{\infty}\left(\frac{k+1}{(\zeta_{m}-1)(\overline{\zeta_{m}}-1)}-\frac{\zeta_{m}+\overline{\zeta_{m}}-2}{(\zeta_{m}-1)^{2}(\overline{\zeta_{m}}-1)^{2}}\right)z^{k}\nonumber \\
 &+ \sum_{k=0}^{\infty}\left(-\frac{1}{(\zeta_{m}-\overline{\zeta_{m}})(\zeta_{m}-1)^{2}(\zeta_{m})^{k+1}}+\frac{1}{(\zeta_{m}-\overline{\zeta_{m}})(\overline{\zeta_{m}}-1)^{2}(\overline{\zeta_{m}})^{k+1}}\right)z^{k}\nonumber \\
= & \sum_{k=0}^{\infty}\left(\frac{k+1}{2-2\cos\frac{2\pi}{m}}+\frac{1}{2-2\cos\frac{2\pi}{m}}+\frac{-(\overline{\zeta_{m}}-1)^{2}(\overline{\zeta_{m}})^{k+1}+(\zeta_{m}-1)^{2}\cdot\zeta_{m}^{k+1}}{(\zeta_{m}-\overline{\zeta_{m}})(\zeta_{m}-1)^{2}(\overline{\zeta_{m}}-1)^{2}(\zeta_{m})^{k+1}(\overline{\zeta_{m}})^{k+1}}\right)z^{k}\nonumber \\
= & \sum_{k=0}^{\infty}\left(\frac{k+2}{2-2\cos\frac{2\pi}{m}}-\frac{\sin\frac{2(k+2)\pi}{m}}{(2-2\cos\frac{2\pi}{m})\sin\frac{2\pi}{m}}\right)z^{k} \nonumber\\
=&\sum_{k=0}^{\infty}\frac{1}{(2-2\cos\frac{2\pi}{m})\sin\frac{2\pi}{m}}\left((k+2)\sin\frac{2\pi}{m}-\sin\frac{2(k+2)\pi}{m}\right)z^{k}.\label{eq: 1-z_final}
\end{align}
Take
\begin{equation}
s_{k}=\frac{1}{(2-2\cos\frac{2\pi}{m})\sin\frac{2\pi}{m}}\cdot\left((k+2)\sin\frac{2\pi}{m}-\sin\frac{2(k+2)\pi}{m}\right).\label{eq:s_k}
\end{equation}
When $|z|<1$, based on (\ref{eq: 1-z_final}), $h_{m}$ can also
be taken as
\begin{align}
h_{m}(z) & =\frac{1}{z^{m-2}}\cdot(z^{m}-1)^{2}\cdot\frac{1}{(z-1)^{2}(z-\zeta_{m})(z-\overline{\zeta_{m}})}\nonumber\\
&=\frac{1}{z^{m-2}}\cdot(z^{2m}-2z^{m}+1)\cdot\frac{1}{(z-1)^{2}(z-\zeta_{m})(z-\overline{\zeta_{m}})}\nonumber \\
 & =\frac{1}{z^{m-2}}\cdot(z^{2m}-2z^{m}+1)\cdot\sum_{k=0}^{\infty}s_{k}z^{k}=\frac{1}{z^{m-2}}\left(\sum_{k=0}^{\infty}s_{k}z^{k}-2\sum_{k=m}^{\infty}s_{k-m}z^{k}+\sum_{k=2m}^{\infty}s_{k-2m}z^{k}\right)\nonumber \\
 & =\frac{1}{z^{m-2}}\left(\sum_{k=0}^{m-1}s_{k}z^{k}+\sum_{k=m}^{2m-1}(s_{k}-2s_{k-m})z^{k}+\sum_{k=2m}^{\infty}(s_{k}-2s_{k-m}+s_{k-2m})z^{k}\right).\label{eq:f_another}
\end{align}
By the definition of $h_m(z)$, we have 
\[
h_m(z)=\frac{(z^m-1)^2/z^m}{\frac{(z-1)^2}{z}\cdot \frac{(z-\zeta_m)(z-\overline{\zeta_m})}{z}}=\frac{z^m+\frac{1}{z^m}-2}{(z+\frac{1}{z}-2)(z+\frac{1}{z}-2(\zeta_m+\overline{\zeta_m}))}.
\]
Therefore, 
\[
h(z)=h(1/z).
\]
Combined with (\ref{eq: f_sum_representation}), we have 
\begin{equation}\label{eq:widetilde ck}
\widetilde{c}_{k}=\widetilde{c}_{-k}=s_{k+m-2}
\end{equation}
for $k=-m+2,\ldots,0$, where $s_{k}$ is defined in (\ref{eq:s_k}).

Furthermore, when $k\in\{-m+2,\dots,0\}$, we have
\begin{align*}
\widetilde{c}_{k} & =s_{k+m-2}=\frac{1}{(2-2\cos\frac{2\pi}{m})\sin\frac{2\pi}{m}}\left((k+m)\sin\frac{2\pi}{m}-\sin\frac{2(k+m)\pi}{m}\right)\\
 & =\frac{1}{(2-2\cos\frac{2\pi}{m})\sin\frac{2\pi}{m}}\sum_{l=1}^{k+m-1}\left(\sin\frac{2\pi}{m}+\sin\frac{2l\pi}{m}-\sin\frac{2\pi(l+1)}{m}\right)\\
 & =\frac{1}{(2-2\cos\frac{2\pi}{m})\sin\frac{2\pi}{m}}\sum_{l=1}^{k+m-1}4\sin\frac{\pi}{m}\sin\frac{l\pi}{m}\sin\frac{(l+1)\pi}{m}\geq\frac{4\sin^2\frac{\pi}{m}\sin\frac{2\pi}{m}}{(2-2\cos\frac{2\pi}{m})\sin\frac{2\pi}{m}}>0.
\end{align*}
The inequality follows from $\frac{l\pi}{m},\frac{(l+1)\pi}{m}\in (0,\pi]$,
when $l\in[1,m-1]$. Therefore, based on  $\widetilde{c}_{k}=\widetilde{c}_{-k}> 0$
for $k\in\{-m+2,\ldots,m-2\}$, we can conclude that
\[
g_{m}(x)=h_{m}\left(e^{\mathrm{i}\frac{2\pi x}{m}}\right)=\sum_{k=-m+2}^{m-2}\widetilde{c}_{k}e^{\mathrm{i}\frac{2k\pi x}{m}}=\widetilde{c}_{0}+\sum_{k=1}^{m-2}2\widetilde{c}_{k}\cos\frac{2k\pi x}{m}=\sum_{k=0}^{m-2}{c}_{k}\cos\frac{2k\pi x}{m}
\]
with positive $\{{c}_{k}\}_{k=0}^{m-2}$, where ${c}_{k}=2\widetilde{c}_{k}$
for $k=1,\ldots,m-2$ and $c_{0}=\widetilde{c}_{0}$.

\textbf{Step 2}: Show that, if $x\in \mathbb{R}$ such that $\|[x]\|_{\mathbb{R}/m\mathbb{Z}}\geq 1$, then $g_{m}(x)\leq0$.

 $h_m$ can be rewritten as
\[
\begin{split}
h_m(z)=&\frac{(z^m-1)^2}{z^m}\cdot\frac{z}{(z-1)^2}\cdot \frac{z}{(z-\zeta_m)(z-\overline{\zeta_m})}\\
=&(z^m+z^{-m}-2)\cdot \frac{1}{z+z^{-1}-2}\cdot \frac{1}{z+z^{-1}-(\zeta_m+\overline{\zeta_m})}.
\end{split}
\]
Then we have
\begin{equation}\label{fm_new}
h_m(e^{\mathrm{i}\frac{2\pi x}{m}})=\frac{\cos(2\pi x)-1}{2(\cos\frac{2\pi x}{m}-1)(\cos \frac{2\pi x}{m}-\cos\frac{2\pi}{m})}.
\end{equation}
Therefore, if  $x\in (-\frac{m}{2},-1)\cup(1,\frac{m}{2})$ and $x\notin \mathbb{Z}$, it obtains that $h_m(e^{\mathrm{i}\frac{2\pi x}{m}})<0$.
Since $g_m(x)=h_m(e^{\mathrm{i}\frac{2\pi x}{m}})$ when $x\notin \mathbb{Z}$ and $g_m$ is continuous, we can directly conclude that $g_m(x)\leq 0$ for any $x$ such that $\|[x]\|_{\mathbb{R}/m\mathbb{Z}}\geq 1$. Therefore, the condition (iii) for $SA_m$ can be achieved. 

\textbf{Step 3}: Prove that $
\{g_m\}_{m\in \mathbb{N}^{+}}^{\#}=1.
$

Since $g_m$ is continuous, we have
\[
g_{m}(0)=\underset{x\rightarrow 0}{\text{lim}}h_m(e^{\mathrm{i}\frac{2\pi x}{m}})=\underset{z\rightarrow 1}{\text{lim}}h_{m}(z)=\frac{m^{2}}{2-2\cos\frac{2\pi}{m}}.
\]
Furthermore, since $c_{0}=\widetilde{c}_{0}=s_{m-2}$ with $\widetilde{c}_{0}$ and $s_{m-2}$
defined in (\ref{eq:widetilde ck}) and (\ref{eq:s_k}). Therefore, we can obtain that
\[
\widehat{g}_m(0)=mc_{0}=m\widetilde{c}_{0}=ms_{m-2}=\frac{m^2\sin\frac{2\pi}{m}}{(2-2\cos\frac{2\pi}{m})\sin\frac{2\pi}{m}}=\frac{m^2}{2-2\cos\frac{2\pi}{m}}.
\]
 Thus we can conclude that
\[
\{g_m\}_{m\in \mathbb{N}^{+}}^{\#}=\underset{m\rightarrow \infty}{\underline{\lim}}\frac{g_m (0)}{\widehat{g}_m (0)}=1.
\]
\end{proof}

\subsection{Proof of Proposition \ref{true_set}}\label{sec: true_set}
\begin{proof}[Proof of Proposition \ref{true_set}]
It is enough to prove that there is no function $f\in \text{SA}$ such that 
\begin{equation}\label{eqn: g_m_temp_temp}
\alpha_mg_m(x)=\sum_{k=-\infty}^{+\infty}f(x+mk),
\end{equation}
for any $m\in \{3,4,\ldots\}$.

If (\ref{eqn: g_m_temp_temp}) holds true, we have
\begin{equation}\label{c_k_temp}
\begin{split}
\int_{-m/2}^{m/2}g_m(x)\cos\left(\frac{2\pi l x}{m}\right)dx&=\frac{1}{\alpha_m}\int_{-m/2}^{m/2}\sum_{k=-\infty}^{+\infty}f(x+mk)\cos\left(\frac{2\pi l x}{m}\right)dx\\
&=\frac{1}{\alpha_m}\sum_{k=-\infty}^{+\infty}\int_{-m/2}^{m/2}f(x+mk)\cos\left(\frac{2\pi l x}{m}\right)dx\\
&=\frac{1}{\alpha_m}\sum_{k=-\infty}^{+\infty}\int_{-m/2+mk}^{m/2+mk}f(x)\cos\left(\frac{2\pi l x}{m}\right)dx\\
&=\frac{1}{\alpha_m}\int_{-\infty}^{\infty}f(x)\cos\left(\frac{2\pi l x}{m}\right)dx=\frac{1}{\alpha_m}\widehat{f}\left(\frac{l}{m}\right)
\end{split}
\end{equation}
for any $l\in \mathbb{Z}$.

According to Step 1 in the proof of Proposition \ref{thm: one_dim_auxiliary}, for any $m=3,4,\ldots$, $g_{m}(x)$ can be rewritten as $g_m(x)=\sum_{k=0}^{m-1}c_{k}\cos\frac{2k\pi x}{m}$ with $c_{m-2}>0$ and $c_{m-1}=0$.
Based on (\ref{c_k_temp}), $c_{m-2}$ and $c_{m-1}$ can also be expressed as
\begin{equation}\label{c_m_temp1}
c_{m-2}=\frac{2}{m}\int_{-m/2}^{m/2}g_m(x)\cos\left(\frac{2(m-2)\pi x}{m}\right)dx=\frac{2}{m\cdot\alpha_m}\widehat{f}\left(\frac{m-2}{m}\right)\neq 0,
\end{equation}
and
\begin{equation}\label{c_m_temp2}
c_{m-1}=\frac{2}{m}\int_{-m/2}^{m/2}g_m(x)\cos\left(\frac{2(m-1)\pi x}{m}\right)dx=\frac{2}{m\cdot \alpha_m}\widehat{f}\left(\frac{m-1}{m}\right)=0.
\end{equation}
Take $m_1=2m$.  According to (\ref{c_m_temp1}), we can get that
\[
\widehat{f}\left(\frac{m_1-2}{m_1}\right)\neq 0.
\]
However, based on (\ref{c_m_temp2}), we also have
\[
\widehat{f}\left(\frac{m_1-2}{m_1}\right)=\widehat{f}\left(\frac{m-1}{m}\right)=0,
\]
which leads to a contradiction.
\end{proof}

\subsection{Proof of Proposition \ref{thm: g_22}}\label{sec: Proof of Proposition 8}
First of all,  we show a technical lemma described as below.
\begin{lem}
\label{lem: technical} Define the function $f$ as
\begin{equation}
f(x,y)=1+\cos\pi x+\cos\pi y+\cos\pi(x+y).\label{eq:f_lemma}
\end{equation}
Then for any $(x,y)\in\mathcal{D}$, we have $f(x,y)\leq0$, where
\begin{equation}\label{eq: X}
\mathcal{D}=\{(x_0+2\mathbb{Z},y_0+2\mathbb{Z})\ :\ |x_0+y_0|\geq1\ \text{and}\ -1\leq x_0,y_0\leq1\}.
\end{equation}
\end{lem}

\begin{proof}
Taking $z_{1}=e^{\mathrm{i}\pi x}$ and $z_{2}=e^{\mathrm{i}\pi y}$,
by some direct calculations,  we have
\begin{align*}
(1+z_{1})(1+z_{2})(1+\frac{1}{z_{1}z_{2}}) & =(1+e^{\mathrm{i}\pi x})(1+e^{\mathrm{i}\pi y})(1+e^{-\mathrm{i}\pi(x+y)})\\
 & =(e^{\mathrm{i\frac{\pi x}{2}}}+e^{\mathrm{-i\frac{\pi x}{2}}})(e^{\mathrm{i\frac{\pi y}{2}}}+e^{\mathrm{-i\frac{\pi y}{2}}})(e^{\mathrm{i\frac{\pi(x+y)}{2}}}+e^{\mathrm{-i\frac{\pi(x+y)}{2}}})\\
 & =8\cos\frac{\pi x}{2}\cdot\cos\frac{\pi y}{2}\cdot\cos\frac{\pi(x+y)}{2}\leq0,
\end{align*}
for any $(x,y)\in\mathcal{D}$.

Furthermore, we can get that
\begin{align*}
(1+z_{1})(1+z_{2})(1+\frac{1}{z_{1}z_{2}}) & =2+(z_{1}+\frac{1}{z_{1}})+(z_{2}+\frac{1}{z_{2}})+(z_{1}z_{2}+\frac{1}{z_{1}z_{2}})\\
 & =2(1+\cos\pi x+\cos\pi y+\cos\pi(x+y))=2f(x,y).
\end{align*}
Therefore, $f(x,y)\leq0$, for any $(x,y)\in\mathcal{D}$.
\end{proof}
Now we begin to prove Proposition \ref{thm: g_22}.

\begin{proof}[Proof of Proposition \ref{thm: g_22}]
According to the definition of $g_2$, we can directly get the condition (ii) for $SA_1$. Now we prove the conditions (i) and (iii) by the following two steps:

\textbf{Step 1}: Show that $g_{2}(\boldsymbol{x})=g_{2}(\boldsymbol{x}+\boldsymbol{v})$, for any $\boldsymbol{x}\in\mathbb{R}^{2}$ and $\boldsymbol{v}\in \Lambda_{2}$.

Denote $\boldsymbol{x}=[x_{1},x_{2}]^{\mathsf{T}}.$ For any $k_1,k_2\in 2\mathbb{Z}$, we have
\begin{align*}
g_{2}(\boldsymbol{x}+k_{1}\boldsymbol{v}_{1}+k_{2}\boldsymbol{v}_{2}) & =g_{2}([x_{1}+k_{1}+\frac{1}{2}k_{2},x_{2}+\frac{\sqrt{3}}{2}k_{2}]^{\mathsf{T}})\\
 & =1+\cos(\boldsymbol{a}_{1}\bm{\cdot} \boldsymbol{x}+\pi k_1)+\cos(\boldsymbol{a}_{2}\bm{\cdot} \boldsymbol{x}+\pi k_2)+\cos(\boldsymbol{a}_{3}\bm{\cdot} \boldsymbol{x}+\pi (k_1+k_2))\\
 & =g_{2}(\boldsymbol{x}).
\end{align*}
It leads to the conclusion.

\textbf{Step 2}: Prove that $g_{2}(\boldsymbol{x})\leq0$, when $\|[\boldsymbol{x}]\|_{\mathbb{R}^2/\Lambda_{2}}\geq 1$.

Denote $\boldsymbol{u}_1=[\pi,0]^{\mathsf{T}}$, $\boldsymbol{u}_2=[0,\pi]^{\mathsf{T}}$ and $\boldsymbol{u}_3=[\pi,\pi]^{\mathsf{T}}$, and take the function $f:\mathbb{R}^{2}\rightarrow\mathbb{R}$ as
\[
f(\boldsymbol{x})=1+\cos(\boldsymbol{u}_1\bm{\cdot} \boldsymbol{x})+\cos(\boldsymbol{u}_2\bm{\cdot} \boldsymbol{x})+\cos(\boldsymbol{u}_3\bm{\cdot} \boldsymbol{x}).
\]
For any fixed $\boldsymbol{x}\in \mathbb{R}^2$, take  $\widetilde{\boldsymbol{x}}=\boldsymbol{A}\boldsymbol{x}$, where $\boldsymbol{A}=\begin{bmatrix}1 & -\frac{1}{\sqrt{3}}\\
0 & \frac{2}{\sqrt{3}}
\end{bmatrix}$. Thus
\[
g_{2}(\boldsymbol{x})=f\left(\widetilde{\boldsymbol{x}}\right).
\]
{Meanwhile,  if $\|[\boldsymbol{x}]\|_{\mathbb{R}^2/\Lambda_2}\geq 1$, then $\boldsymbol{x}$ can be rewritten as
\[
\boldsymbol{x}=\boldsymbol{x}_0+\boldsymbol{v},
\]
for certain $\boldsymbol{v}\in \Lambda_2$ and certain $\boldsymbol{x}_0\in \mathcal{P}(\Lambda_2)-(\boldsymbol{v}_1+\boldsymbol{v}_2)$ with $\|\boldsymbol{x}_0\|_2\geq 1$.
Thus $\widetilde{\boldsymbol{x}}=\boldsymbol{A}\boldsymbol{x}=\boldsymbol{A}(\boldsymbol{x}_0+\boldsymbol{v})=\widetilde{\boldsymbol{x}}_0+\widetilde{\boldsymbol{v}}$,  where $\widetilde{\boldsymbol{v}}\in 2\mathbb{Z}^2$ and $\widetilde{\boldsymbol{x}}_0=[\widetilde{x}_{0,1},\widetilde{x}_{0,2}]^{\mathsf{T}}$ such that $\|\boldsymbol{A}^{-1}\widetilde{\boldsymbol{x}}_0\|_2\geq 1$, $-1\leq \widetilde{x}_{0,1}<1$ and $-1\leq \widetilde{x}_{0,2}<1$. By some direct calculations, we have  $|\widetilde{x}_{0,1}+\widetilde{x}_{0,2}|\geq 1$. Then $\widetilde{\boldsymbol{x}}\in \mathcal{D}$,
where $\mathcal{D}$ is defined in (\ref{eq: X}).
According to Lemma \ref{lem: technical}, we have $f(\widetilde{\boldsymbol{x}})\leq0$,
for any $\widetilde{\boldsymbol{x}}\in\mathcal{D}$. Therefore,   for any $\boldsymbol{x}$ such that $\|[\boldsymbol{x}]\|_{\mathbb{R}^2/\Lambda_2}\geq 1$, we
have $g_{2}(\boldsymbol{x})\leq0$.}

Furthermore, by Definition \ref{param} and Remark \ref{g remark}, we have
\begin{equation}
g_2^{\#}=\frac{g_2(\boldsymbol{0})}{\widehat{g}_2 (\boldsymbol{0})}=\frac{g_{2}(\boldsymbol{0})}{{c}_{0}|\Lambda_{2}|}=\frac{2}{\sqrt{3}}.\label{eq: delta_upper}
\end{equation}
The proof is completed. 
\end{proof}

\subsection{Proof of Proposition \ref{prop: 3dim}}\label{dim_3: sec}
\begin{proof}[Proof of Proposition \ref{prop: 3dim}]
According to the definition of $g$, we can directly get the condition (ii) for $SA_1$. Now we prove the conditions (i) and (iii) by the following two steps:

\textbf{Step 1}:  Show that $g(\boldsymbol{x})=g(\boldsymbol{x}+\boldsymbol{v})$, for any $\boldsymbol{x}\in \mathbb{R}^2$ and $\boldsymbol{v}\in \Lambda$. 

Denote $\boldsymbol{x}=[x_{1},x_{2},x_{3}]^{\mathsf{T}}.$ For any $k_1,k_2,k_3\in \mathbb{Z}$, we have 
\begin{align*}
g(\boldsymbol{x}+k_{1}\boldsymbol{v}_{1}+k_{2}\boldsymbol{v}_{2}+k_{3}\boldsymbol{v}_3) & =g\left([x_{1}+\sqrt{2}k_1,
x_{2}+\sqrt{2}k_2,
x_{3}+\sqrt{2}k_3]^{\mathsf{T}}\right)\\
 & =1+\cos(\boldsymbol{a}_{1}\bm{\cdot} \boldsymbol{x}+2\pi k_1)+\cos(\boldsymbol{a}_{2}\bm{\cdot} \boldsymbol{x}+2\pi k_2)+\cos(\boldsymbol{a}_{3}\bm{\cdot} \boldsymbol{x}+2\pi k_3)\\
 & =g(\boldsymbol{x}).
\end{align*}
It leads to the conclusion. 

\textbf{Step 2:} Prove that $g(\boldsymbol{x})\leq 0$, when $\|[\boldsymbol{x}]\|_{\mathbb{R}^3/\Lambda}\geq 1$.

 If $\|\boldsymbol{x}\|_{\mathbb{R}^3/\Lambda}\geq 1$, then $\boldsymbol{x}=[x_1,x_2,x_3]^{\mathsf{T}}$ can be rewritten as 
\[
\boldsymbol{x}=\boldsymbol{x}_0+\boldsymbol{v},
\]
where $\boldsymbol{x}_0\in [-\frac{\sqrt{2}}{2},\frac{\sqrt{2}}{2}]^3$ such that $\|\boldsymbol{x}_0\|_2\geq 1$ and $\boldsymbol{v}\in \Lambda$. 

Take $\boldsymbol{u}=[u_1,u_2,u_3]^{\mathsf{T}}$, where $u_1=\sqrt{2}\pi x_1$, $u_2=\sqrt{2}\pi x_2$ and $u_3=\sqrt{2}\pi x_3$. Then $g$ can be rewritten as $\widetilde{g}$, that is, 
\[
\widetilde{g}(\boldsymbol{u})=1+\cos(u_1)+\cos(u_2)+\cos(u_3). 
\]
According to the definition of $\boldsymbol{x}$, $\boldsymbol{x}_0$ and $\cos(z)=\cos(-z)$ for any $z\in \mathbb{R}$, in order to get the conclusion, it is sufficient to prove that $\widetilde{g}(\boldsymbol{u})\leq 0$, when $
\boldsymbol{u}\in \{\boldsymbol{u}\in [0,\pi]^3\ :\ \|\boldsymbol{u}\|_2\geq \sqrt{2}\pi\}. 
$
Since $\widetilde{g}$ is a continuous function, it is enough to prove that, when 
\[
\boldsymbol{u}\in \{\boldsymbol{u}\in [0,\pi]^3\ :\ \|\boldsymbol{u}\|_2> \sqrt{2}\pi\},
\]
we have $\widetilde{g}(\boldsymbol{u})\leq 0$. 
If we can get that 
\begin{equation}\label{eqn: belong}
\small
\{\boldsymbol{u}\in [0,\pi]^3\ :\ \|\boldsymbol{u}\|_2> \sqrt{2}\pi\}\subset \{\boldsymbol{u}\in [0,\pi]^3\ :\ \|\boldsymbol{u}\|_1> 2\pi\},
\end{equation}
it is enough to show $\widetilde{g}(\boldsymbol{u})\leq 0$, when $\boldsymbol{u}\in [0,\pi]^3$ and $\|\boldsymbol{u}\|_1\geq 2\pi$. In this case, as  $\pi \geq u_3\geq 2\pi -u_1-u_2$,  we have 
\[
\begin{split}
\widetilde{g}(\boldsymbol{u})=&1+\cos(u_1)+\cos(u_2)+\cos(u_3)\\
\leq & 1+\cos(u_1)+\cos(u_2)+\cos(2\pi-u_1-u_2)\\
=&1+\cos(u_1)+\cos(u_2)+\cos(u_1+u_2),
\end{split}
\]
where $0\leq u_1,u_2\leq \pi$ and $u_1+u_2\geq \pi$.  Taking $u_1=\pi x$ and $u_2=\pi y$ in Lemma \ref{lem: technical}, we can directly get that $\widetilde{g}(\boldsymbol{u})\leq 0$, which meets the conclusion. 

Now the only thing left is to prove (\ref{eqn: belong}). It is equivalent to show that when $\boldsymbol{u}\in \Omega:=\{\boldsymbol{u}\in [0,\pi]^3\ :\  \|\boldsymbol{u}\|_1\leq 2\pi\}$, then $\|\boldsymbol{u}\|_2\leq \sqrt{2}\pi$ holds true. Take $\boldsymbol{u}_0=[0,0,0]^{\mathsf{T}}$, 
\[
 \boldsymbol{u}_1=[\pi,\pi,0]^{\mathsf{T}}, \ \boldsymbol{u}_2=[\pi,0,\pi]^{\mathsf{T}}, \  \boldsymbol{u}_3=[0,\pi,\pi]^{\mathsf{T}}, \  \boldsymbol{u}_4=[\pi,0,0]^{\mathsf{T}}\  \boldsymbol{u}_5=[0,\pi,0]^{\mathsf{T}}\  \text{and}\ \boldsymbol{u}_6=[0,0,\pi]^{\mathsf{T}}.
\]
Since $\|\boldsymbol{u}_i\|_2\leq \sqrt{2}\pi$, $i=0,\ldots,6$, and the convex hull of $\{\boldsymbol{u}_i\}_{i=0}^{6}$ is $\Omega$, we can immediately arrive at the result.

Furthermore, by Definition \ref{param} and Remark \ref{g remark}, we have
\begin{equation}
g^{\#}=\frac{g(\boldsymbol{0})}{\widehat{g} (\boldsymbol{0})}=\frac{g(\boldsymbol{0})}{{c}_{0}|\Lambda|}=\sqrt{2}.\end{equation}
The proof is completed.
\end{proof}

\end{document}